\documentclass[leqno,11pt]{article}
\usepackage{microtype}

\usepackage{amsmath,amsthm,amssymb,mathrsfs} 
\usepackage[T1]{fontenc}
\usepackage{ae,aecompl}
\usepackage{times}
\usepackage[english]{babel}
\usepackage[colorlinks,pagebackref,hypertexnames=false]{hyperref}
\usepackage[numeric,backrefs]{amsrefs}
\usepackage{bookmark}

\numberwithin{equation}{section}

\newcommand{\rG}{{\rm G}}

\newcommand{\Ad}{\mathrm{Ad}}

\renewcommand{\epsilon}{\varepsilon}

\newcommand{\dvol}{\mathop\mathrm{dvol}\nolimits}

\newcommand{\ev}{\mathrm{ev}}

\def\<{\mathopen{}\left<}
\def\>{\right>\mathclose{}}
\def\({\mathopen{}\left(}
\def\){\right)\mathclose{}}

\newtheorem{theorem}{Theorem}

\newtheorem{definition}{Definition}
\newtheorem{example}{Example}

\newtheorem{lemma}{Lemma}

\newtheorem{proposition}{Proposition}
\newtheorem{remark}{Remark}

\numberwithin{equation}{section}

\makeatletter
\newcommand{\subjclass}[2][1991]{%
  \let\@oldtitle\@title%
  \gdef\@title{\@oldtitle\footnotetext{#1 \emph{Mathematics subject classification.} #2}}%
}
\newcommand{\keywords}[1]{%
  \let\@@oldtitle\@title%
  \gdef\@title{\@@oldtitle\footnotetext{\emph{Key words and phrases.} #1.}}%
}
\makeatother

\author{Goncalo Oliveira \\ Duke University}

\title{$G_2$-monopoles with singularities (examples)}

\date{September 2016}

\begin{document}
\maketitle
\footnotetext{\emph{MSC(2010)} 57R57 (Primary), 53C29, 53C38, 53C07} 
\footnotetext{\emph{Key words} Monopole, $G_2$-manifold}

\begin{abstract}
$G_2$-monopoles are solutions to gauge theoretical equations on $G_2$-manifolds. If the $G_2$-manifolds under consideration are compact, then any irreducible $G_2$-monopole must have singularities. It is then important to understand which kind of singularities $G_2$-monopoles can have. We give examples (in the noncompact case) of non-Abelian monopoles with Dirac type singularities, and examples of monopoles whose singularities are not of that type.\\
We also give an existence result for Abelian monopoles with Dirac type singularities on compact manifolds. This should be one of the building blocks in a gluing construction aimed at constructing non-Abelian ones.
\end{abstract}


\section{Introduction}

A $G_2$-manifold is a seven dimensional manifold $X^7$ equipped with a Riemannian metric $g$ whose holonomy lies in $G_2$. Similarly this can be encoded in a stable $3$ form $\varphi$, which determines a Riemannian metric whose holonomy is contained in $G_2$ if and only if $\varphi$ is both closed and coclosed. We let $\psi=\ast \varphi$, where $\ast$ is the Hodge $\ast$ operator associated with the metric $g$, and will refer to a $G_2$-manifold as the pair $(X^7, \varphi)$. A $G_2$-manifold is said to be irreducible if the holonomy of the Riemannian metric is equal to $G_2$.\\
We now introduce $G_2$-monopoles. For that, let $G$ be a compact semisimple Lie group and $P$ a principal $G$-bundle over $M$. Denote by $\mathfrak{g}_P$ the bundle associated with the adjoint representation and equip it with an $Ad$-invariant metric.

\begin{definition}\label{def:monopole}
A pair $(A, \Phi)$ where $A$ is a connection on $P$ and $\Phi \in \Omega^0(M, \mathfrak{g}_P)$ called an Higgs field is said to be a \textbf{$G_2$-monopole} (or simply monopole) if
$$\ast(F_A \wedge \psi) = d_A \Phi,$$
where $F_A$ is the curvature of $A$ and $d_A \Phi$ the covariant derivative of $\Phi$ with respect to the connection induced by $A$ on $\mathfrak{g}_P$.
\end{definition}

Most of the interest in monopoles comes from Donaldson and Segal's, \cite{Donaldson2009}, suggestion that it may be possible to define an invariant of $G_2$ manifolds by counting monopoles. The authors have further suggested that such monopoles may be somehow related to certain coassociative submanifolds. In the noncompact case evidence towards such a relation was found in \cite{Oliveira2014}.\\
If one is given a monopole $(A,\Phi)$, then the Bianchi identity and the fact that $\psi$ is closed imply that $d_A \ast d_A \Phi =0$. Hence, $\Delta_A \Phi=0$ and so 
$$\Delta \frac{\vert \Phi \vert^2}{2} = \langle \Phi, \Delta_A \Phi \rangle -\vert d_A \Phi \vert^2 \leq 0.$$
This means that $\vert \Phi \vert^2$ is subharmonic. Hence, if $M$ is compact, then $\vert \Phi \vert$ is constant and so $ d_A \Phi =0$ and $F_A \wedge \psi=0$, i.e. $A$ is a so called $G_2$-instanton, which is actually reducible in the case when $\Phi \neq 0$. Monopoles in $G_2$ manifolds may also be relevant for $M$-theory compactifications in manifolds of special holonomy. See for example, pages $78$--$84$ in \cite{Acharya04}, regarding type IIA string theory on $\mathbb{R}^4 / \Gamma_{ADE} \times S^3$.\\
This leads one to either let $X$ be noncompact, or allow the monopoles to have singularities. The goal of this paper is to initiate the study of monopoles on $G_2$-manifolds with a specified kind of singularities. There is a special class of singularities, which we call Dirac type singularities. The idea is to consider monopoles on the complement of a suitable submanifold of $X$. In a similar related work Yuanqui Wang in \cite{Wang2016} have recently studied $G_2$-monopoles with point like singularities.\\
Recall that an oriented real $4$-dimensional submanifold $N \subset X$ is said to be coassociative if it is calibrated by the $4$-form $\psi$, i.e. $\psi \vert_N = dvol_N$, where $dvol_N$ is the volume form associated with the restriction of the metric $g$ to $N$. Equivalently, a coassociative submanifold can be defined by $\varphi \vert_N =0$. Now let $N =N_1 \cup ... \cup N_k$, where the $N_i$ are disjoint, compact, connected and embedded coassociative submanifolds. Denote by $r_i=dist(\cdot,N_i):X \rightarrow \mathbb{R}^+_0$, the geodesic distance to $N_i$. 

\begin{definition}\label{def:DiracSingularity}
Let $N \subset X$ be as above and $M=X \backslash N$. Then a pair $(A, \Phi)$ is said to be a monopole on $X$ with \textbf{Dirac type singularities} along $N$ if the following hold
\begin{enumerate}
\item $(A, \Phi)$ is a monopole on $P$ over $M$ as in definition \ref{def:monopole}.
\item For each $i \in \lbrace 1,...,k \rbrace$, there is $k_i \in \mathbb{Z}$ such that
$$\lim_{r_i \rightarrow 0} r_i \vert \Phi \vert = \lim_{r_i \rightarrow 0} r_i^2 \vert F_A \vert = k_i,$$
and the monopole is then said to have \textbf{charge} $k_i$ along $N_i$.
\end{enumerate}
Moreover, if the gauge group is $G=S^1$ then we shall say the monopole is a \textbf{Dirac monopole}.
\end{definition}

Note that, according to this definition, a Dirac monopole is an Abelian monopole with Dirac type singularities. Moreover, a monopole which smoothly extends over $N$ to the whole $X$ has Dirac type singularities with charge $0$ along $N$.

\begin{remark}\label{rem:ev}
Notice that coassociative submanifolds are of codimension $3$. Let $\mathbb{S}^2_i$ be any fiber of the unit $2$-sphere bundle $\mathbb{S}^2(N_i)$ normal to $N_i$, and
$$ev_i: H^2(X \backslash N , \mathbb{Z}) \rightarrow \mathbb{Z}$$
the evaluation map on the class of $[\mathbb{S}^2_i] \in H_2(\mathbb{S}^2(N_i), \mathbb{Z})$. If $(A, \phi)$ is a Dirac monopole as in the definition above, with charge $k_i \in \mathbb{Z}$ along $N_i$. Then $ev_i(c_1(L))=k_i$, where $L$ is the complex line bundle associated to the $S^1$-bundle carrying the monopole $(A,\phi)$. 
\end{remark}

In this paper we give some examples of monopoles with singularities, most of which are of Dirac type. We shall now give an outlook of these results which serves as a guide to how the paper is organized. 

\subsection*{Main results and outlook of the singularity zoo}

We start in section \ref{sec:Dirac} with the most basic examples, motivating our definition \ref{def:DiracSingularity} of Dirac type singularities. These are the Dirac monopole on $\mathbb{R}^7$ and some nonAbelian monopoloes with Dirac type singularities, both of which are obtained by extending monopoles on $\mathbb{R}^3$ to $\mathbb{R}^7$ by making them translation invariant. Their singularities are located along $\lbrace 0 \rbrace \times \mathbb{R}^4$, which is coassociative. Then in \ref{sec:BS} we turn to the Bryant-Salamon $G_2$-manifolds \cite{Bryant1989}. After recalling the examples in \cite{Oliveira2014} of Dirac monopoles, in section \ref{sec:BSDiracType} we give the first examples of non-Abelian monopoles with Dirac type singularities on the Bryant-Salamon $G_2$-manifolds. These are invariant under a suitable group action and obtained by analyzing the resulting ODE's, which are found in \cite{Oliveira2014}.

\begin{theorem}\label{thm:BSDyracType}
Let $M= \mathbb{S}^4$ (resp. $\mathbb{CP}^2$), then there is an $SU(2)$(resp. $SO(3)$)-bundle $P \rightarrow \Lambda^2_-(M)$ equipped with a real two parameter family of irreducible monopoles with singularities along the zero section $M$. Moreover, the singularities are such that
\begin{equation}
\lim_{r \rightarrow 0 } r\vert \Phi \vert = \lim_{r \rightarrow 0 } r^2 \vert F_A \vert = 1.
\end{equation}
\end{theorem}

Then, in section \ref{sec:Ugly} by analyzing ODE's obtained for monopoles on another bundle, we find examples of monopoles with singularities which are worse than of Dirac type. See theorem \ref{prop:Ugly} for a more precise statement, which we state here as

\begin{theorem}\label{thm:Ugly}
There are $SO(3)$ bundles $P_{\alpha_1}$, $P_{\alpha_3}$ over $\Lambda^2_-(\mathbb{CP}^2) \backslash \mathbb{CP}^2$, on which there is a real two parameter family of $SU(3)$-invariant, irreducible $G_2$-monopoles with gauge group $SO(3)$ (possibly not defined on the whole of $\Lambda^2_-(\mathbb{CP}^2) \backslash \mathbb{CP}^2$). This family of monopoles is such that
\begin{itemize}
\item For any open neighborhood $U$ of the zero section, there exists an open set $V \supset \overline{U}$ such that a member of this family of monopoles is defined on $V \backslash \overline{U}$.

\item If in this family there is a monopole $(A, \Phi)$ defined in a neighborhood of the zero section, with the zero section removed, then $\Phi$ diverges exponentially along it.
\end{itemize}
\end{theorem}

Our definition of Dirac type monopoles deliberately excludes exponentially diverging singularities. However, we conjecture that if the monopole $(A,\Phi)$ lives on a $SO(3)$-bundle $P$ over $X \backslash N$, induced from restricting a bundle $\tilde{P}$ over $X$, then the singularities of $(A,\Phi)$ are at most of Dirac type.\\
The final section \ref{sec:General} gives an abstract construction of Dirac monopoles on a compact $G_2$-manifold. Namely we shall prove that

\begin{theorem}\label{th:Dirac}
Let $(X, \varphi)$ be a compact, irreducible $G_2$ manifold, i.e. with full holonomy $G_2$. Let $N = N_1 \cup \ldots \cup N_k$ be a disjoint union of compact, connected and embedded coassociative submanifolds of $(X, \varphi)$ and $M=X \backslash N$. Then, for all $\alpha \in H^2(M , \mathbb{Z})$, there is Dirac monopole $(A, \Phi)$ on $X$, defined on a line bundle $L \rightarrow M$ with $c_1(L)=\alpha$. Moreover, such monopole has charge $\ev_i(\alpha)$ along each $N_i$.
\end{theorem}

\begin{remark}
For any $(n_1,\ldots , n_k) \in \mathbb{Z}^k$ such that $\sum_{i=1}^k n_i [N_i] =0 \in H_4(X, \mathbb{Z})$ there does exist $\alpha \in H^2(M, \mathbb{Z})$ with $\ev_i(\alpha)=n_i$. Hence, for any such $(n_1, \ldots , n_k)$ we can construct a Dirac monopole with charge $n_i$ along $[N_i]$.
\end{remark}

It remains open the problem of constructing non-Abelian monopoles with Dirac type singularities on a compact $G_2$-manifold. Studying such monopoles may lead to a numerical invariant of compact $G_2$ manifolds, possibly related to their coassociative geometry. The author expects theorem \ref{th:Dirac} to provide one of the building blocks of such a construction, and intends to come back to this problem in future work.

\section*{Acknowledgments}

I would like to thank Robert Bryant, Edgar Costa and Mark Stern for discussions. I would also like to thank an anonymous referee for very helpful comments and suggestions, in particular for having detected an error in a first version of the proof of proposition $1$. I would also like to thank the Max-Planck-Institute in Bonn, for hosting me during the summer of 2015.

\section{Singular monopoles on $\mathbb{R}^7$}\label{sec:Dirac}

We shall consider $\mathbb{R}^7 = \mathbb{R}^4_y \times \mathbb{R}^3_x$ with the flat $G_2$-structure, for which
$$\psi = dy^0 \wedge dy^1 \wedge dy^2 \wedge dy^3 - \frac{1}{2} \epsilon_{ijk} dx^{i}\wedge dx^{j} \wedge \Omega_k,$$
where $\Omega_i= dy^0 \wedge dy^i -  \frac{1}{2} \epsilon_{ijk} dy^j \wedge dy^k$ is a basis for $\Lambda^2_- \mathbb{R}^4$, i.e. the anti-self-dual $2$-forms in $\mathbb{R}^4$.

\begin{lemma}\label{lem:reduction}
Monopoles $(A, \Phi)$ on $\mathbb{R}^7$ invariant under translations in the $\mathbb{R}^4_y$ directions are in one to one correspondence with solutions to the following equations in $\mathbb{R}^3$
\begin{eqnarray}\nonumber
d_a \Phi - [B, \Psi] & = &  \ast_3  ( F_a - \frac{1}{2} [B \wedge B] ) \\ \nonumber
d_a \Psi + [B, \Phi] & = &  \ast_3 d_a B \\ \nonumber
d_a^* B + [\Psi, \Phi] & = & 0.
\end{eqnarray}
where $d_a$ is a connection on a principal bundle $P$, $\Phi,\Psi \in \Omega^0(\mathbb{R}^3 , \mathfrak{g}_P)$ called Higgs fields, $B \in \Omega^1(\mathbb{R}^3 , \mathfrak{g}_P)$ and $\ast_3$ the Hodge-star associated with the Euclidean metric in $\mathbb{R}^3$. In particular, setting $B = \Psi =0$, the equations above reduce to
$$d_a \Phi =  \ast_3  F_a.$$
Hence any Bogomolnyi monopole on $\mathbb{R}^3$ can be lifted to a $G_2$-monopole on $\mathbb{R}^7$.
\end{lemma}
\begin{proof}
If the bundle $P$ is pulled back from $\mathbb{R}^3$, then any connection on $P \rightarrow \mathbb{R}^3$ can be written as $d_A= d_a + b$, where $d_a$ is the pullback of a connection from $\mathbb{R}^3$ and $b \in \Omega^0(\mathbb{R}^7, \Lambda^1 \mathbb{R}^4_y \otimes \mathfrak{g}_P)$. If we further suppose that the pair $(A, \Phi)$ is invariant by translations along the $\mathbb{R}^4$ directions then $b$ is constant along the $y$ directions, i.e. we could interpret it as being the pullback of $b \in \Omega^0(\mathbb{R}^3_x, \Lambda^1 \mathbb{R}^4_y \otimes \mathfrak{g}_P)$ and the same holds for $\Phi$ which we can interpret as being pullback from $\Phi \in \Omega^0(\mathbb{R}^3_x, \mathfrak{g}_P)$. Then, $F_A = F_a + d_a b + \frac{1}{2}[b \wedge b]$ and $d_A \Phi = d_a \Phi + [b, \Phi]$ and the $G_2$-monopole equations turn into
\begin{eqnarray}\nonumber
0 & = &  \ast \left( (F_a + d_a b + \frac{1}{2}[b \wedge b]) \wedge \psi \right) - d_a \Phi - [b, \Phi] .
\end{eqnarray}
Splitting this equation into its components in $\Lambda^1 \mathbb{R}^4_y \otimes \mathfrak{g}_P$ and $\Lambda^1 \mathbb{R}^3_x \otimes \mathfrak{g}_P$ we get
\begin{eqnarray}\nonumber
d_a \Phi & = &  \ast_3  F_a + \gamma( b , b ) \\ \nonumber
[b, \Phi] & = &  \ast( d_a b \wedge \psi ).
\end{eqnarray}
where $\ast_3$ denotes the $3$-dimensional Hodge star operator in the Euclidean $\mathbb{R}^3_x$ and $\gamma(\cdot , \cdot)$ is a certain multilinear pairing. Using the notation $b=b^0 dy^0 + \sum_{i=1}^3 b^i dy^i$, where $b^0, b^i \in \Omega^0(\mathbb{R}^3_x , \mathfrak{g}_P)$ we have $\gamma (b,b) = ( \epsilon_{ijk} [b^i,b^j]- [b^0, b^k]) dx^k$ and the equations above turn into
\begin{eqnarray}\nonumber
d_a \Phi & = &  \ast_3  F_a - ( \epsilon_{ijk} [b^i,b^j]- [b^0, b^k]) dx^k \\ \nonumber
[b^0, \Phi] & = &  -\sum_{i=1}^3 \nabla^a_{\partial_i} b^i \\ \nonumber
[b^k, \Phi] & = &  \nabla^a_{\partial_k} b^0  + \epsilon_{ijk} \nabla^a_{\partial_i} b^j,
\end{eqnarray}
and $\nabla^a_{\partial_k}$ denotes covariant differentiation with respect to $d_a$ in the direction of $\frac{\partial}{\partial x^k}$. It is now easy to see that by setting $B=b^i dx^i$ and $\Psi = - b^0$ we obtain the equations in the statement.
\end{proof}

\begin{remark}
There is an elegant way to write the equations in lemma \ref{lem:reduction}. In fact we can define the complexified connection and Higgs field $\nabla_A = \nabla_a + i B$ and $\Gamma = \Phi + i \Psi \in \Omega^0(\mathbb{R}^7 , \mathfrak{g}_P^{\mathbb{C}})$, where $ \mathfrak{g}_P^{\mathbb{C}} = \mathfrak{g}_P \otimes_{\mathbb{R}} \mathbb{C}$ is the complexified adjoint bundle. Then
$$F_A= F_a  + i d_a B -  \frac{1}{2}[B \wedge B] \ , \ d_A \Gamma = d_a \Phi + i d_a \Psi + i [B, \Phi] - [B, \Psi].$$
and it is straightforward to see that the first and third equations in lemma \ref{lem:reduction} are, respectively, the real and imaginary parts of the complexified Bogomolnyi equation
\begin{equation}\label{eq:ComplexMonopole}
d_A \Gamma = \ast F_A.
\end{equation}
As for the second equation in the lemma, it can be written as $d_a^* (iB) - \frac{1}{2} [\Gamma , \overline{\Gamma}]=0$.\\
In fact the equations above form a system of elliptic equations that can be written in any $3$ manifold. In $\mathbb{R}^3$ it seems possible that one can equip the moduli space of solutions with an hyperk\"ahler structure.\\
It would be interesting to check if there is a Kempf-Ness type result relating the moduli space of solutions to the total system of equations, and that of solutions to equation \ref{eq:ComplexMonopole} (modulo the action of the complexified gauge group, in this latter case).
\end{remark}

Now we shall use lemma \ref{lem:reduction} to start with a singular monopole on $\mathbb{R}^3$ and lift it to a singular monopole on $\mathbb{R}^7$.

\begin{example}
We shall consider $G= \mathbb{S}^1$. Then we can think of $A$ as a connection on a complex line bundle and $\Phi$ as a function, as the adjoint bundle is trivial. The Bogomolnyi equations turn into $\ast d \Phi = F_A$ and the Bianchi identity $d F_A=0$ implies that $\Delta \Phi =0$. To search for a monopole with a singularity at the origin we consider an harmonic function on $\mathbb{R}^3_x \backslash \lbrace 0 \rbrace$ decaying at infinity. These are of the form $\Phi = m -\frac{k}{\vert x \vert}$, where $m,k \in \mathbb{R}$, and we define
$$F_A = \ast d \Phi =  \frac{k}{\vert x \vert^2} \dvol_{\mathbb{S}^2}.$$
This $F_A$ is obviously closed on $\mathbb{R}^3 \backslash \lbrace 0 \rbrace$ and in order to be the curvature of a connection on a complex line bundle we need $\frac{1}{2 \pi}[F_A] \in H^2(\mathbb{R}^3 \backslash \lbrace 0 \rbrace , \mathbb{R})$ to be integral, i.e. $k \in \mathbb{Z}$. If this is the case then $F_A$ is the curvature of a connection $A$ on $H^k$, the radial extension of the Hopf bundle. The monopole $(A, \Phi)$ is known as the charge $k$ and mass $m$ Dirac monopole.\\
Now, we split $\mathbb{R}^7 = \mathbb{R}^4_y \times \mathbb{R}^3_x$ and pullback $H^k$ together with $(A, \Phi)$ to $\mathbb{R}^7$. It follows from lemma \ref{lem:reduction} that $(A, \Phi)$ is a monopole on $\mathbb{R}^7$. This is singular along the coassociative submanifold $N= \mathbb{R}^4 \times \lbrace 0 \rbrace$ and $H$ is the pullback of the Hopf bundle from the spheres in the normal bundle to $N$. Moreover, notice that if $r=\vert x \vert$ denotes the distance to $N$, then
$$r\vert \Phi \vert = r^2 \vert F_A \vert = k.$$
\end{example}

\begin{example}
There are also non-Abelian examples of singular monopoles. In fact, equations $A.6$ in the Appendix to \cite{Oliveira2014}, a two parameter family of explicit irreducible monopoles with gauge group $SU(2)$ on $\mathbb{R}^3$ is given. These are spherically invariant, i.e. only depend on $r= \vert x \vert$. One can easily check that for these monopoles
$$\vert \Phi \vert= \Big\vert \frac{1}{r} - C \coth (Cr + D) \Big\vert \ , \ \vert F \vert = \Big\vert \frac{1}{r^2} + \frac{\sqrt{2}C}{\sinh(D)} \frac{1}{r} + O(1) \Big\vert$$
where $C,D$ are two real parameters such that $CD >0$. Then, from this one can check that
$$\lim_{r \rightarrow 0 } r\vert \Phi \vert = \lim_{r \rightarrow 0 } r^2 \vert F_A \vert = 1.$$
\end{example}

\section{Singular monopoles on the Bryant-Salamon $G_2$-manifolds}\label{sec:BS}

The Bryant-Salamon $G_2$ manifolds \cite{Bryant1989} having compact coassociative submanifolds are the total spaces of anti-self-dual $2$ forms $\Lambda^2_-(M)$ on a self-dual, Einstein four manifold $M$ with positive scalar curvature. These are either $\mathbb{S}^4$ or $\mathbb{CP}^2$ with $g_M$ being respectively the round and the Fubini-Study metrics. In either case, the zero section is the unique compact coassociative submanifold. Let $\pi: \Lambda^2_-(M) \rightarrow M$ denote the projection (this is the twistor projection), then the Bryant-Salamon metric can be written as
$$g = f^2(s) g_{\mathbb{R}^3} +f^{-2}(s(r)) \pi^*  g_M,$$
where $g_{\mathbb{R}^3}$ is the Euclidean metric along the fibers, $f(s)=(1+s^2)^{-1/4}$ and $s$ is the Euclidean distance along the fibers to the zero section. Then, the geodesic distance to zero section in the metric $g$ is $r(s)=\int_0^s f(t) dt$ and using it we can write the metric as
$$g = dr^2 + s^2(r) f^2(s(r)) g_{\mathbb{S}^2} +f^{-2}(s) \pi^*  g_M,$$
where $g_{\mathbb{S}^2}$ is the round metric in the normal spheres to $M$. We now define the function
\begin{equation}\label{eq:h}
h^2(t) = Vol(r^{-1}(t))=s^2(t) f^{-2}(s(t)),
\end{equation}
then there is a unique function $G$ such that
$$\frac{d G}{d t } = \frac{1}{h^2(t)} \ , \ \lim_{t \rightarrow \infty} G(t) =0.$$
This function is well defined in all of $\mathbb{R}^+$ being unbounded at the origin. Moreover Taylor expanding $G$ we can see that $G(t) = - \frac{1}{t} + O(1)$ for $t \ll 1$ and $G(t) = - \frac{c}{t^{5}} +  O(t^{-6})$ where $c >0$ for $t \gg 1$. It is an easy computation to show that $G \circ r : \Lambda^2_-(M) \backslash M \rightarrow \mathbb{R}$ satisfies $\Delta G =0$, i.e. is harmonic with respect to the Bryant-Salamon metrics.

\subsection{Dirac Monopoles, i.e. Abelian Examples}\label{sec:BSDirac}

\begin{example}
For $M = \mathbb{S}^4$, the complement of the zero section $\Lambda^2_-(\mathbb{S}^4) \backslash \mathbb{S}^4$ is topologically a cone over $\mathbb{CP}^3$ and so $H^2(\Lambda^2_-(\mathbb{S}^4) \backslash \mathbb{S}^4, \mathbb{Z}) \cong \mathbb{Z}$. Let $L \rightarrow \Lambda^2_-(\mathbb{S}^4) \backslash \mathbb{S}^4$ be the complex line bundle such that $c_1(L)$ is the generator. Then it is proven in \cite{Oliveira2014}, proposition $7$ that for all $k \in \mathbb{Z}$ and $m \in \mathbb{R}$ there is a monopole $(A_k , \Phi_{m,k})$ on $L^k$ such that $\Phi_{m,k}= m + kG$. These monopoles are singular at the zero section, which is coassociative and it is easy to check that
$$\lim_{r \rightarrow 0} r \vert \Phi_{m,k} \vert = \lim_{r \rightarrow 0} r^2 \vert F_{A_k} \vert = k.$$
\end{example}

\begin{example}
The complement of the zero section in $\Lambda^2_-(\mathbb{CP}^2)$ retracts onto $\mathbb{F}_2$, the manifold of full flags in $\mathbb{C}^3$. There is an isomorphism $H^2(\Lambda^2_-(\mathbb{CP}^2) \backslash \mathbb{CP}^2, \mathbb{Z}) \cong \mathbb{Z}^2$, under which the image of the map $\pi^* : H^2(\mathbb{CP}^2, \mathbb{Z}) \rightarrow H^2(\mathbb{F}_2, \mathbb{Z})$ is precisely the diagonal in $\mathbb{Z}^2$. Also in \cite{Oliveira2014}, before proposition $9$, it is proven that for all $(n,l) \in \mathbb{Z}^2$ and $m \in \mathbb{R}$ there is a monopole $(A_{(n,l)} , \Phi_{m,(n,l)})$ on a line bundle $L^{(n,l)}$ with $c_1(L^{(n,l)})=(n,l)$. In this case $\Phi_{m,(n,l)}= m + (l-n)G$ and we can check by Taylor expanding $G$ close to the zero section that
$$\lim_{r \rightarrow 0} r \vert \Phi_{m,(n,l)} \vert = \lim_{r \rightarrow 0} r^2 \vert F_{A_{(n,l)}} \vert = l-n.$$
In particular, when $l=n$ so that $L$ is pulled back from $\mathbb{CP}^2$ via the twistor projection $\pi$, $\Phi_{m, (n,n)}=m$ is constant and the connection $A_{(n,n)}$ is the pullback of a self-dual connection on $\mathbb{CP}^2$.
\end{example}

\begin{remark}
We remark that in fact, in the two examples above the connections $A_{k}$ and $A_{(n,l)}$ are the pullback to $\mathbb{R}^+ \times \mathbb{CP}^3 \cong \Lambda^2_-(\mathbb{S}^4) \backslash \mathbb{S}^4$, respectively $\mathbb{R}^+ \times \mathbb{F}_2 \cong \Lambda^2_-(\mathbb{CP}^2) \backslash \mathbb{CP}^2$, of pseudo-Hermitian-Yang-Mills connections for the homogeneous nearly K\"ahler structures on $\mathbb{CP}^3$ and $\mathbb{F}_2$.
\end{remark}

\subsection{Non-Abelian monopoles with Dirac type singularities}\label{sec:BSDiracType}

On the Bryant Salamon metrics on $\Lambda^2_-M$ we have already seen examples of Dirac monopoles (which recall are Abelian). It remains the question of whether non-Abelian monopoles with nontrivial "Dirac type" singularities along the zero section exist. In fact, they do exist both for $M=\mathbb{CP}^2$ and $\mathbb{S}^4$ by theorem \ref{thm:BSDyracType} in the introduction, which we shall now restate and prove.

\begin{proposition}(Theorem \ref{thm:BSDyracType})
Let $M= \mathbb{S}^4$ (resp. $\mathbb{CP}^2$), then there is an $SU(2)$(resp. $SO(3)$)-bundle $P \rightarrow \Lambda^2_-(M)$ equipped with a real two parameter family of irreducible monopoles with singularities along the zero section $M$. Moreover, the singularities are such that
\begin{equation}\label{eq:LimitBS}
\lim_{r \rightarrow 0 } r\vert \Phi \vert = \lim_{r \rightarrow 0 } r^2 \vert F_A \vert = 1.
\end{equation}
\end{proposition}
\begin{proof}
In \cite{Oliveira2014} the monopole equations are reduced to ODE's under a symmetry assumption. We recall here the ODE's from propositions $6$ and $10$ in \cite{Oliveira2014}
\begin{eqnarray}\label{eq:sODE1}
\dot{\phi} & = & \frac{b^2-1}{2h^2}\\ \label{eq:sODE2}
\dot{b} & = & 2b\phi,
\end{eqnarray}
where $h$ is as in equation \ref{eq:h}. Therefore, to produce irreducible monopoles with singularities it is enough to produce solutions to those ODE's satisfying the required properties. In particular, for the connection to be irreducible it is enough for $b$ to be nonzero, which will be the case for the solutions we construct below. Let $t_0 \in \mathbb{R}^+$. The standard theorem for existence and uniqueness of solutions guarantee that for any initial condition at $t_0$ there is a unique (real analytic) solution to \ref{eq:sODE1} and \ref{eq:sODE2}. Hence, we construct the $2$-parameter family of solutions parametrized by $t_0 \in \mathbb{R}^+$ and $b_0 \in (0,1)$, given by the initial conditions
$$\phi(t_0)=0, \ \ b(t_0)=b_0 \in (0,1).$$
We will now prove that for any of these solutions $b(t) \in (0,1)$, for all $t \in \mathbb{R}^+$. Suppose not, then there is $t_1 \in \mathbb{R}^+$ such that either $b(t_1)=0$ or $b(t_1)=1$.
\begin{itemize}
\item In the first case, i.e. if $b(t_1)=0$, then $\dot{b}(t_1)=0$ and
$$\dot{\phi}(t_1) = -\frac{1}{2h^2(t_1)} < \infty.$$
By continuing to differentiate the equation $\dot{b}=2 \phi b$ we can prove that $b^{(k)}(t_1)=0$ for all $k \in \mathbb{N}$. Hence, as $b$ is real analytic it must vanish identically on all $\mathbb{R}^+_0$.

\item Notice that as $\phi(t_0)=0$, $\dot{b}(t_0)=0$ and so $b$ has a critical point at $t_0$. By differentiating the ODE's we are left with a single second order ODE for either $\phi$ or $b$, the former of which is
$$b \ddot{b} = \dot{b}^2 + \frac{b^2}{h^2}(b^2 -1) .$$
We already know that $b$ remains positive, and we can use this equation to infer the possible values of $b$ at its critical points. Namely, if $b$ has a critical point with $b < 1$, then that point must be a maximum, while if $b > 1$ it must be a minimum. Hence, the critical point of $b$ at $t_0$ is a maximum as $b(t_0) \in (0,1)$. Now suppose that $b$ crosses $1$, then by continuity there must be $t_2 \in \mathbb{R}^+$ such that $b(t_2) \in (0,1)$ and $b$ has a minimum at $t_2$, which is a contradiction.
\end{itemize}
We conclude that $b$ remains bounded for all time $t \in \mathbb{R}^+$ and so the only way $\phi$ can blow up is at the singularities of $\frac{1}{2h^2(t)}$, the only of which is at $t=0$.\\
To prove the behavior claimed in equation \ref{eq:LimitBS}, notice that since $b(t) \in (0,1)$ for $t \in \mathbb{R}^+$, $\dot{\phi}= \frac{1}{2 h^2}(b^2-1)<0$. Moreover, as $\phi(t_0)>0$ we have $\phi(t)>0$, for all $t \in (0, t_0]$. Then, since $\dot{b}=2\phi b$, $\dot{b}(t)>0$ for $t \in (0, t_0]$ and so $b(t)<b(t_0)<1$ for $t \in (0,t_0]$. This, together with the fact that $b(t)>0$ for such $t$'s, proves that there is a constant $c_1>1$, such that
\begin{equation}\label{eq:phiFirstBound}
- \frac{c_1^{-1}}{t^2} \geq \dot{\phi}(t) \geq -\frac{c_1}{t^2},
\end{equation}
where we used the fact that $h^2(t)=t^2 +O(t^3)$, for small $t$. Integrating this for $t \in (0,t_0]$
\begin{eqnarray}\nonumber
\phi(t)-\phi(t_0) & \geq &  -\int_{t_0}^t \frac{c_1}{s^2} ds = \frac{c_1}{t} - \frac{c}{t_0},
\end{eqnarray}
and similarly for an upper bound. We conclude that so $\frac{c_1^{-1}}{t} + c_3 > \phi(t) > \frac{c_1}{t} + c_2$ for small $t$ and some $c_2,c_3 \in \mathbb{R}$.\\
Inserting this back into $\dot{b}=2\phi b$ we conclude that $ c_4 t^{2c_1^{-1}} \geq b(t) \geq c_4^{-1} t^{2c_1}$ for some $c_4 >0$ and all $t \in (0, t_0]$. This shows that $b$ continuously extends to the origin with $b(0)=0$.\\
We can now go back to the ODE for $\phi$ and improve our previous bounds to $-\frac{1}{2t^2} + c_5 t^{4c_1^{-1}-2} \geq \dot{\phi} \geq - \frac{1}{2t^2} + c_6 t^{4c_1-2}$, for small $t$. In fact from these we can actually infer that the constant $c_1$ in the bound \ref{eq:phiFirstBound} could have been taken to take values in $(1,4)$. Hence, once integrated, the previous bounds yield
$$\frac{1}{2t} + c_5' \geq \phi(t) \geq \frac{1}{2t} + c_6' ,$$
for some $c_5', c_6' \in \mathbb{R}$ and all $t$ sufficiently small. It is now immediate to conclude that the limiting behavior in equation \ref{eq:LimitBS} holds.

\end{proof}

\subsection{An example of worse than Dirac singularities}\label{sec:Ugly}

We now focus on the Bryant-Salamon metric on $\Lambda^2_-( \mathbb{CP}^2)$ on which we have proved monopoles with nontrivial Dirac type singularities exist. It remains the question of whether there are monopoles with singularities which are not of this type. In this section we show these indeed exist, and give an example of a singular monopole on $\Lambda^2_-( \mathbb{CP}^2)$, whose singularities are worse than the ones we have seen so far.\\
As already remarked above, the sphere bundle in $\Lambda^2_-( \mathbb{CP}^2)$, i.e. the twistor space of $\mathbb{CP}^2$, is the flag manifold $\mathbb{F}_2$. This is homogeneous and $SU(3)$ acts transitively with isotropy the maximal torus $T^2$. The Serre spectral sequence for the fibration $SU(3) \rightarrow \mathbb{F}_2$ gives $H^2( \mathbb{F}_2 , \mathbb{Z}) \cong H^1(T^2, \mathbb{Z})$, which we can further identify with the integral weight lattice in $(\mathfrak{t}^2)^*$. An explicit way to unravel through this identification using Chern classes to make the identification is as follows. Given an integral weight $\alpha \in (\mathfrak{t}^2)^*$ we construct the line bundle on $\mathbb{F}_2$
$$L_{\alpha} = SU(3) \times_{e^{\alpha}, T^2} \mathbb{C}.$$
Now let $1 \in SU(3)$ be the identity and $\mathfrak{m}\subset \mathfrak{su}(3)$ be a reductive complement to the Cartan subalgebra generated by the isotropy, i.e. $\mathfrak{su}(3) = \mathfrak{t}^2 \oplus \mathfrak{m}$ with $[\mathfrak{t}^2 , \mathfrak{m}] \subset \mathfrak{m}$ (for example, we can let $\mathfrak{m}$ be the real part of the root spaces). Then, we extend $\alpha$, first to $\mathfrak{su}(3)^*$ by letting it vanish on $\mathfrak{m}$, and secondly to $\Omega^1(SU(3), i\mathbb{R})$ by left translations. It is now easy to see that $\alpha$ equips $L_{\alpha}$ with a connection and so its first Chern class $ \frac{i}{2 \pi}[d \alpha] \in H^2(\mathbb{F}_2 , \mathbb{Z})$ gives the corresponding element in the second cohomology induced by $\alpha$. Back to the connection $\alpha$, it is usually called the canonical invariant connection on $L_{\alpha}$ and is uniquely determined by $\mathfrak{m}$.\\
We shall now turn to the construction of $SO(3)$-bundles over $\mathbb{F}_2$, carrying interesting invariant connections. These are constructed by composing the homomorphism $e^{ \alpha }: T^2 \rightarrow \mathbb{S}^1$ with the embedding of $ \mathbb{S}^1 \hookrightarrow SO(3)$ as the maximal torus, then setting
$$P_{\alpha} = SU(3) \times_{(e^{ \alpha}, T^2)} SO(3).$$
These $SO(3)$-bundles are in fact reducible to the circle bundles inducing $L_{\alpha}$ and can be equipped with the induced connections $ \alpha \in  \Omega^1(SU(3), \mathfrak{so}(3))$ viewed as left invariant $1$-forms in $SU(3)$ with values in $\mathfrak{so}(3)$ by embedding $i \mathbb{R} \hookrightarrow \mathfrak{so}(3)$. These induced connections are also $SU(3)$-invariant and it follows from Wang's theorem, \cite{Wang1958}, that other invariant connections are in $1$ to $1$ correspondence with morphisms of $T^2$-representations
$$\Lambda : ( \mathfrak{m}, \Ad ) \rightarrow (\mathfrak{so}(3), \Ad \circ e^ {\alpha} ).$$
Decomposing these into irreducible components $ \mathfrak{m}  \cong \mathbb{C}_{\alpha_1} \oplus \mathbb{C}_{\alpha_2} \oplus \mathbb{C}_{\alpha_3}$, where $\alpha_1, \alpha_2, \alpha_3$ are the positive roots of $SU(3)$, while $\mathfrak{so}(3) \cong \mathbb{R}_0 \oplus \mathbb{C}_{\alpha}$. Hence it follows from Schur's lemma that such morphisms of representations exist if and only if $\alpha$ is one of the roots, in which case $\Lambda$ restricts to the corresponding root space as an isomorphism onto $\mathbb{C}_{\alpha} \subset \mathfrak{so}(3)$ and vanishes in all other components. If $\alpha = \alpha_i$ we shall denote these by $\Lambda_i$. Then, notice that fixing a basis of $\mathfrak{m}$ and a basis of $\mathfrak{so}(3)$ (i.e. a gauge) each $\Lambda_i$ is determined up to a constant.\\
We turn now to the problem of constructing monopoles on the bundles $P_{\alpha}$. In \cite{Oliveira2014} the monopole equation in each of these cases is analyzed and it is shown that smooth solutions exist only for $\alpha = \alpha_2$, where $\mathbb{C}_{\alpha_2}$ is the image of $(\pi_2)^*: H^2(\mathbb{CP}^2, \mathbb{Z}^2) \rightarrow H^2(\mathbb{F}_2 , \mathbb{Z}^2)$, where $\pi_2 : \mathbb{F}_2 \rightarrow \mathbb{CP}^2$ is the twistor projection. In that case these monopoles can be completely classified as in theorem $6$. In the other cases, the underlying $SO(3)$-bundles do not extend over the zero section (lemma $5$) and it is shown in proposition $11$ that no smooth solutions exist. To understand the result and the non-smooth solutions arising from solving the $SU(3)$-invariant monopole equations it is convenient to proceed as follows. Take $\alpha= \alpha_3$ (the case $\alpha = \alpha_1$ is similar), and extend the bundle and the connection to the complement of the zero section, i.e. to $\mathbb{R}^+_r \times \mathbb{F}_2$. Now the connection $\alpha + \Lambda_3 (r)$ can be seen as an element of $\Omega^1(\mathbb{R}^+ \times SU(3), \mathfrak{so}(3))$. Invariant Higgs fields are in correspondence with $SU(3)$-invariant maps $\mathbb{R}^+  \times SU(3) \rightarrow  \mathfrak{so}(3)$, which are also $T^2$-equivariant, with $T^2$ acting by right translations on $SU(3)$ and by $\Ad \circ e^{ \alpha_3}$ on $\mathfrak{so}(3)$. These two conditions force such Higgs fields to be in correspondence with functions $i \phi : \mathbb{R}^+ \rightarrow i \mathbb{R}$ composed with the map $i \mathbb{R} \rightarrow \mathfrak{so}(3)$ induced by the maximal torus embedding. Then, in \cite{Oliveira2014} the invariant monopole equations are computed. They reduce to ODE's and in terms of $\phi$ and $b^2=2 s^2( r ) f^{-2}(r) \vert \Lambda_{3} \vert^2$ these are
\begin{eqnarray}\label{eq:ODE1}
\dot{\phi} & = & \frac{1 + b^2}{2h^2}\\ \label{eq:ODE2}
\dot{b} & = & 2b\phi,
\end{eqnarray}
where the dot denotes differentiation with respect to $r$ and $h,f,s$ are defined in the previous section, when the Bryant-Salamon metrics were introduced. For all solutions to the ODE's above, $\phi$ is unbounded as $1+b^2 \geq 1$ and $h(0)=0$. We shall now show that this singularity is worse than the singularities we have seen so far. Namely we shall prove theorem \ref{thm:Ugly} in the introduction

\begin{theorem}\label{prop:Ugly}(Theorem \ref{thm:Ugly})
There is a real $2$-parameter family of $SU(3)$-invariant, irreducible monopoles on the bundles $P_{\alpha_1}$ and $P_{\alpha_3}$ over $\Lambda^2_-(\mathbb{CP}^2) \backslash \mathbb{CP}^2$ (possibly not defined on the whole  $\Lambda^2_-(\mathbb{CP}^2) \backslash \mathbb{CP}^2$). Moreover, these monopoles have the following properties
\begin{itemize}
\item For any $\epsilon>0$, there is a monopole in this family which is defined in a neighborhood of $r^{-1}(\epsilon)$.

\item If there is a monopole $(A, \Phi)$ which is defined in a neighborhood of the zero section, with the zero section removed, then there is $\delta>0$ such that
\begin{equation}\label{eq:BadSing2}
\vert e^{-\frac{\delta}{r^2}} \Phi(r) \vert 
\end{equation}
is unbounded in that neighborhood of the zero section.
\end{itemize}
\end{theorem}
\begin{proof}
The monopole resulting from evolving the ODE's is reducible if and only if $b=0$ identically and $\phi$ solves $\dot{\phi}=\frac{1}{2h^2}$, hence we shall exclude this case from the analysis. Let $R>0$ be a fixed positive number, and denote by $(\phi(r),b(r))$ the solutions to the ODE system \ref{eq:ODE1}-\ref{eq:ODE2} for $r\leq R$, which at $r=R$ are valued $(\phi(R), b(R))$ for some $b(R) \neq 0$ (in order to exclude the reducible one). These parametrize the $2$-parameter family alluded in the statement. Then, either
\begin{itemize}
\item $(\phi(r), b(r))$ explodes before $r=0$, in which case the monopole is only defined away from the zero section. Notice that from the ODE's \ref{eq:ODE1}-\ref{eq:ODE2}, if either the fields $b$ or $\phi$ explodes at some $r \in (0,R]$, the other one also explodes at that same $r$. Further notice that, given $\epsilon>0$, one can make $R=\epsilon$ and so the first item in the statement holds.
\item $(\phi(r), b(r))$ exists on the whole interval $(0,R]$, in which case one must prove the second item in the statement.
\end{itemize}
We are then reduced to consider the second case above. In the remainder of this proof we shall use $c_i$'s to denote positive constants, which can be chosen so that the bounds claimed are true. From the first ODE above, \ref{eq:ODE1} it follows that for $R>r$
$$\phi(R) - \phi(r) = \int_r^R \frac{1 + b^2(t)}{2h^2(t)} dt \geq \int_r^R \frac{1}{2h^2(t)} dt.$$
Then, we Taylor expand $h(t)= t + O(t^3)$ close to the origin and we conclude that there is $c_1>0$ such that $\phi(R) - \phi(r) \geq -c_1 + \frac{1}{2r}$, which we can rearrange to $\phi(r) \leq \phi(R)+c_1-\frac{1}{2r}$. We plug this into the second ODE, i.e. equation \ref{eq:ODE2} which then gives $\frac{d}{dr}(\log (b^2(r))) \leq 4( \phi(R)+c_1-\frac{1}{2r})$. Then, we integrate this to
$$\log \left( \frac{b^2(R)}{b^2(r)} \right) \leq 4(\phi(R)+c_1)(R-r) -2 \log \left( \frac{R}{r} \right),$$
which we can rearrange to $b^2(r) \geq b^2(R) \frac{R^{2}}{r^{2}} e^{-4(\phi(R)+c_1)(R-r)}$. Putting this back into equation \ref{eq:ODE1} and integrating we obtain now
\begin{eqnarray}\nonumber
\phi(R) - \phi(r) & = & \int_r^R \frac{1 + b^2(t)}{2h^2(t)} dt \\ \nonumber
& \geq & -c_1 + \frac{1}{2r} + b^2(R) \int_r^R \frac{1}{2h^2(t)} \frac{R^{2}}{t^{2} } e^{-4(\phi(R)+c_1)(R-t)}  dt \\ \nonumber
& \geq & -c_1 + \frac{1}{2r} + c_2 \int_r^R \frac{1}{t^{4}} dt \\ \label{eq:ineqIntermediate}
& \geq & -c_3 + \frac{1}{2r} + \frac{c_4}{r^{3}},
\end{eqnarray}
where $c_2,c_3,c_4>0$ are constants. We can now insert this into the ODE \ref{eq:ODE1} ands improve the bound on $b(r)$ to again improve the bound on $\phi(r)$. In fact, to prove our claim, it is enough to iterate this only once more. From inserting inequality \ref{eq:ineqIntermediate} into the ODE \ref{eq:ODE2} once again, we obtain $\frac{d}{dr}(\log (b^2(r))) \leq 4( \phi(R)+c_3-\frac{c_5}{r^{3}})$ and integrating
\begin{eqnarray}
\log \left( \frac{b^2(R)}{b^2(r)} \right) & \leq &  c_6 -\frac{c_5}{r^{2}},
\end{eqnarray}
by possibly redefining the constant $c_5$. This shows that $b^2(r) \geq b^2(R) e^{-c_6 +\frac{c_5}{r^{2}}}$, which when inserted into equation \ref{eq:ODE1} gives
\begin{eqnarray}\nonumber
\phi(R) - \phi(r) & = & \int_r^R \frac{1 + b^2(t)}{2h^2(t)} dt \\ \nonumber
& \geq & -c_1 + \frac{1}{2r} + c_7 \int_r^R \frac{e^{\frac{c_5}{t^{2}}}}{t^2} \\ \nonumber
& \geq & -c_8 + \frac{1}{2r} + \frac{c_7}{c_5} r e^{\frac{c_5}{r^2}}+ \frac{c_7}{c_5} \int_r^{R} e^{c_5/t^2} dt,
\end{eqnarray}
which diverges exponentially as $r \rightarrow 0$ from the right. More precisely, in order to prove the claim in the statement we check that
\begin{eqnarray}\nonumber
\lim_{r\rightarrow 0} \vert \Phi(r) e^{-\frac{\delta}{r^2}} \vert \geq c_8 \lim_{r \rightarrow 0} \frac{\int_{r}^{R} e^{c_5/t^2} dt}{e^{\delta/r^2}} = \frac{c_8}{2 \delta} \lim_{r \rightarrow 0} r^3 e^{\frac{c_5-\delta}{r^2}} = \infty,
\end{eqnarray}
if $\delta < c_5$.

\end{proof}

\section{Singular monopoles on compact $G_2$ manifolds}\label{sec:General}

In this section we prove theorem \ref{th:Dirac}, which gives sufficient conditions for the existence of Dirac monopoles, i.e. Abelian monopoles with Dirac type singularities, on a compact $G_2$ manifold. Then, we give a toy non-Abelian example with Dirac type singularities, where the underlying $G_2$-structure is not torsion free. We finish by setting the problem of constructing non-Abelian monopoles with Dirac type singularities, which we hope to address in the future.

\subsection*{Proof of theorem \ref{th:Dirac}}

We recall here the statement of theorem

\begin{theorem}\label{th:Dirac2}(Theorem \ref{th:Dirac})
Suppose $(X, \varphi)$ is a compact, irreducible $G_2$ manifold, i.e. it has full holonomy $G_2$. Let $N = N_1 \cup \ldots \cup N_k$ be a disjoint union of compact, connected and embedded coassociative submanifolds of $(X, \varphi)$. Then, for all $\alpha \in H^2(M , \mathbb{Z})$, there is an Abelian monopole $(A, \Phi)$ on $X$ with Dirac type singularities along $N$, defined on a line bundle $L \rightarrow M$ with $c_1(L)=\alpha$. Moreover, such monopole has charge $\ev_i(\alpha)$ along each $N_i$.
\end{theorem}

We divide the proof into $4$ steps:\\

\textbf{Step $1$:} There is an exact sequence 
\begin{equation}\label{eq:exactSeq}
H^2(X, \mathbb{R}) \xrightarrow{i} H^2(M, \mathbb{R}) \xrightarrow{j} H^0(N) \xrightarrow{\delta} H^3(X, \mathbb{R}).
\end{equation}
In particular, for all $\alpha \in H^2(M, \mathbb{Z})$ we have $\sum_{i=1}^k \ev_i (\alpha)[N_i]=0 \in H_4(X, \mathbb{Z})$.\\

The sequence above and its exactness follow from the long exact sequence for the pair $(X,M)$, which yields the sequence $H^{\ast}(X,M) \rightarrow H^{\ast}(X, \mathbb{R}) \rightarrow H^{\ast}(M, \mathbb{R})$. Then, excision and Thom's isomorphism theorem gives $H^{\ast}(X,M) \cong H^{\ast-3}(N)$ and so the sequence \ref{eq:exactSeq}. The claim that $\sum_{i=1}^k \ev_i (\alpha)[N_i]$ vanishes for all $\alpha \in H^2(M, \mathbb{Z})$ is then immediate from this exact sequence as $\delta \circ j( \alpha )=0$.\\

\textbf{Step $2$}: Let $k \in \mathbb{Z}$, $\alpha \in H^2(M, \mathbb{Z})$ and $r_i: M \rightarrow \mathbb{R}^+$ denote the geodesic distance to $N_i$, for $i \in \lbrace 1, \ldots , k \rbrace$. Then, we shall prove that on $M$ there is a closed $2$-form $F'$, and a real valued function $\phi$ on $M$, such that
\begin{equation}\label{eq:MonEq0}
\ast (F' \wedge \psi) = d \phi,
\end{equation}
and $\lim_{r_i \rightarrow 0} r_i\phi = \ev_i([F'])= \ev_i(\alpha)$.\\

This part of the proof is motivated by Hitchin's work \cite{Hitchin99}. We start by noticing that $[N_{\alpha}]=\sum_{i=1}^k \ev_i (\alpha)[N_i]$ is homologous to zero by the first step. Hence, the harmonic representative of $PD_X[N_{\alpha}]=0$ is $0$ and we can solve the PDE
$$\Delta H = \sum_{i=1}^k \ev(\alpha_i)\delta_{N_i},$$
for a $4$-current $H$, which we identify with a $3$-form with distributional coefficients. Moreover, as $N$ is compact $\Delta d H=0$. Hence $dH$ vanishes as it is a global harmonic and exact $4$-form. Then we define $F'=d^*H$ (which is the connection $2$-form of the gerbe on the open set $M=X \backslash N$). Since $dF'=dd^*H=0$ on $M$, $F'$ is closed and we shall now check that $F'$ satisfies $\ast (F' \wedge \psi) = d \phi$, where $\phi$ is the function such that $\frac{\phi}{7} \varphi= \pi_1(H)$.\\
Recall that as $N$ is coassociative, $\varphi \vert_N =0$. Then for all $\eta \in \Omega^1(X)$,
$$\delta_{N_i} \wedge \varphi (\eta) = \int_{N_i} \eta \wedge \varphi =0.$$
This shows that $\delta_{N_i} \wedge \varphi=0$, or in other words $\pi_7(\delta_{N_i}) =0 \in \Lambda^3_7$. Hence, $\pi_7(\Delta H) = 0$, and as in a $G_2$-manifold the Laplacian preserves the type decomposition, $\pi_7(H)$ is harmonic. Moreover, as $X$ is irreducible, there can be no parallel $1$-forms, \cite{Bryant1989}, and the Bochner-formula implies that $\pi_7(H)=0$. As a consequence, the equation $dH=0$ turns into $d \pi_1(H) = - d \pi_{27}(H)$ and writing $\pi_1 (H) = -a \varphi$, for some function $a$ on $M$ (which extends to $X$ as a $7$-current), this is
\begin{equation}\label{dH27}
d\pi_{27}(H)=  da \wedge \varphi.
\end{equation}
Then $d^* \pi_{27}H = \pi_7 d^* \pi_{27} H + \pi_{14}d^* \pi_{27} H$, and using the identities $\pi_7 d^* \pi_{27} H = - \frac{1}{3} \ast (\ast ( \ast d \pi_{27} H \wedge \varphi ) \wedge \psi )$, and $\ast (\ast (d a \wedge \varphi) \wedge \varphi) = -4 d a$, together with equation \ref{dH27} gives
\begin{eqnarray}\nonumber
d^* H & = & \pi_{14}d^* \pi_{27} H -  \frac{1}{3} \ast \left( \ast \left( \ast d \pi_{27}H \wedge \varphi \right) \wedge \psi \right) - d^* (a \varphi) \\ \nonumber
& = &  \pi_{14}d^* \pi_{27} H +\frac{4}{3} \ast \left( d a \wedge \psi \right)  + \ast \left( d a \wedge \psi \right) \\ \nonumber
& = & \pi_{14} d^* \pi_{27} H +  \frac{7}{3}  \ast \left( d a \wedge \psi \right) .
\end{eqnarray}
At this point we define $\phi = 7a$, and recall that $F'=d^* H$. Then, as $\Omega^2_{14}$ is the kernel of wedging with $\psi$ and $\ast ( \ast (d \phi \wedge \psi ) \wedge \psi) = 3 d \phi$, we obtain
\begin{equation}
\ast \left( F' \wedge \psi \right) = d \phi,
\end{equation}
which proves our equation \ref{eq:MonEq0}. It remains to prove that the function $\phi$ has the claimed limiting behavior around each $N_i$. Moreover, such function is harmonic on $M$ and can be extended to $X$ as a current satisfying $\Delta \phi= \sum_{i=1}^k \ev_i(\alpha) \delta_{N_i} \wedge \psi$. It follows then that on $B_{\epsilon}(N_i)$,
$$\phi= \frac{\ev_i(\alpha)}{r_i}  +O(1).$$
The only thing left to check is that $\ev_i([F'])=\ev_i(\alpha)$, for all $i \in \lbrace 1,\ldots , k \rbrace$. This follows immediately from evaluating $[F']$ along the cycles generated by the fibers of $\mathbb{S}^2(N_i)$. For $x_i \in N_i$ the fiber above it is a two sphere $\mathbb{S}^2_{x_i}$, bounding a disk $D$ intersecting the zero section. Then, it follows from Stokes' theorem that
$$\int_{\mathbb{S}^2_{x_i}} F'= \int_D dF' = \ev_i(\alpha),$$
as $dF= \Delta H = \sum_{i=1}^k \ev_i(\alpha) \delta_{N_i}$.\\

For this $2$-form $F'$ obtained in step $2$ to define a connection on a line bundle bundle $L$ over $M$, with $c_1(L)=\alpha$, we would need $[F']= \alpha \in H^2(M, \mathbb{Z})$. This may not be the case in general. However, as we shall see in the next step, it is always possible to change $F'$ to another $2$-form $F$, in such a way that equation \ref{eq:MonEq0} still holds for $F$ instead of $F'$, and $[F]= \alpha $.\\

\textbf{Step $3$}: Let $\beta \in H^2(X, \mathbb{R})$, we prove that the harmonic representative $b$ of $\beta$ is such that $\pi_7(b)=0$, i.e. $b \wedge \psi=0$.\\

The proof is a consequence of $(X, \varphi)$ being irreducible, as in this case it can have no parallel $1$-forms. Since $g_{\varphi}$ is Ricci-flat, the B\"ochner formula on $1$-forms gives $\nabla^* \nabla = \Delta$ and since $X$ is compact there can be no harmonic $1$-forms also. This proves that $H^2_{7}(X, \mathbb{R})=0$ and so the harmonic representative of any cohomology class has no component along the standard $7$-dimensional $G_2$-representation.\\

\textbf{Step $4$}: We finish the proof by putting all the previous steps together. Let $\alpha \in H^2(M, \mathbb{Z})$ as in the hypothesis. Then, we construct $F'$ and $\phi$ using step $2$, these satisfy 
\begin{equation}\nonumber
\ast (F' \wedge \psi) = d \phi,
\end{equation}
with $\lim_{r_i \rightarrow 0} r_i\phi = \ev_i(\alpha)$, and $F'$ a closed $2$-form on $M$, such that $ev_i([F'])=\ev_i(\alpha)$. In other words, using the exact sequence in the first step the class $[F'] \in H^2(M, \mathbb{R})$ is such that $j([F'])=j(\alpha)$. And so by exactness $[F']-\alpha = i(\beta)$ for some $\beta \in H^2(X, \mathbb{R})$. By the third step, the harmonic representative $b$ of the class $\beta$ satisfies $b\wedge \psi =0$. Hence we define $F=F'+b$, which we immediately check satisfies
$$\ast (F \wedge \psi) = d \phi$$
and $[F]= \alpha \in H^2(M, \mathbb{Z})$. It then follows as an immediate application of the Poincar\'e lemma that $F$ is the curvature of a connection $A$ in a line bundle $L \rightarrow M$ with $c_1(L)=\alpha$. And the pair $(A, \Phi)$ is a monopole on $L$ which extends to a monopole with Dirac type singularities of charge $\ev_i(\alpha)$ along each $N_i$.

\begin{remark}
For any $(n_1,\ldots , n_k) \in \mathbb{Z}^k$ such that $\sum_{i=1}^k n_i [N_i] =0 \in H_4(X, \mathbb{Z})$, the exactness of the sequence in the first step of the previous proof shows that there does exist $\alpha \in H^2(M, \mathbb{Z})$ with $\ev_i(\alpha)=n_i$. Hence, we can construct a Dirac monopole with charge $n_i$ along $[N_i]$.
\end{remark}

\subsection*{Singular monopoles on $\mathbb{T}^4 \times \mathbb{S}^3$}\label{sec:Coclosed}

It is difficult to come up with an example of a non-Abelian singular monopole on a compact $G_2$-manifold. However, it is crucial to get some examples in order to test ideas. So far, the best we can do is a toy example on a compact manifold equipped a coclosed (but not closed $G_2$-structure).\\
Let $X=\mathbb{T}^4 \times \mathbb{S}^3$ and denote by $\lbrace d\theta^a \rbrace_{a=0}^{3}$ the standard coframing of the torus and by $\lbrace \eta^i \rbrace_{i=1}^{3}$ the usual $SU(2)$-invariant coclosed coframing of $\mathbb{S}^3$, i.e. $d \eta^i = -2 \epsilon_{ijk} \eta^{jk}$, where $\eta^{j} \wedge \eta^{k}=\eta^{jk}$. Then we shall define the $G_2$-structure
$$\varphi= \eta^{123} + \eta^1 \wedge \Omega^1 + \eta^2 \wedge \Omega^2 + \eta^3 \wedge \Omega^3,$$
where the $\Omega^i$'s form a basis for the anti-self-dual $2$-forms on $\mathbb{T}^4$, for concreteness take $\Omega^1 = d \theta^0 \wedge d \theta^1 - d \theta^2 \wedge d \theta^3$, $\Omega^2 = d \theta^0 \wedge d \theta^2 - d \theta^3 \wedge d \theta^1$ and $\Omega^3 = d \theta^0 \wedge d \theta^3 - d \theta^1 \wedge d \theta^2$. We can easily check that these $G_2$ structures can never be closed. However, it is also easy to see that the induced $4$-form
$$\psi= \theta^{0123} - \eta^{23} \wedge \Omega^1 - \eta^{31} \wedge \Omega^2 - \eta^{12} \wedge \Omega^3,$$
is closed if and only if the $\Omega_i$'s are closed, which indeed they are.

\begin{remark}
In fact, the above construction can be done more generally with $X=M^4 \times M^3$, where $\overline{M}^4$ is hyperk\"ahler and $M^3$ is any $3$-manifold (one can prove that any $3$-manifold admits a coclosed framing as I learned from Robert Bryant).\\
Also, we remark that there is an $h$-principle for coclosed $G_2$-structures, \cite{Crowley2012}, and any spin $7$-manifold admits one such.
\end{remark}

As the $G_2$-structure is only coclosed, the equations for coassociative submanifolds are overdetermined. However, in our example these do exist. In fact, any coassociative submanifold for this $G_2$-structure is of the form $\mathbb{T}^4 \times \lbrace p \rbrace$, where $p \in S^3$. Now we let $N= \mathbb{T}^4 \times \lbrace \infty \rbrace$ and $M= X \backslash N \cong \mathbb{T}^4 \times \mathbb{R}^3$, by stereographically projecting from $\lbrace \infty \rbrace$. Then, in \cite{Pauly98} the author constructs a monopole with gauge group $SU(2)$ on $\mathbb{S}^3$. This has a Dirac type singularity at $\lbrace \infty \rbrace \in \mathbb{S}^3$ and the Higgs field $\Phi$ vanishes in its antipodal point $\lbrace 0 \rbrace$. We pull this monopole back to $X$, then it follows from lemma \ref{lem:reduction} that we obtain a $G_2$-monopole $(A, \Phi)$ on $X$ with Dirac type singularities along $N= \mathbb{T}^4 \times \lbrace \infty \rbrace$ and $\Phi^{-1}(0)=  \mathbb{T}^4 \times \lbrace 0 \rbrace$. Moreover, let $r$ be the distance to $N$, then this Dirac monopole satisfies \cite{Pauly98} 
$$\lim_{r \rightarrow 0} r \vert \Phi \vert = \lim_{r \rightarrow 0} r^2 \vert F_A \vert=1.$$
We now interpret these singularities in terms of the rest of the discussion in this section. In the example under consideration, the moduli space of coassociative submanifolds is parametrized by $\mathbb{S}^3$. Recall that $N= \mathbb{T}^4 \times \lbrace \infty \rbrace$ and $M=X \backslash N \cong \mathbb{T}^4 \times \mathbb{R}^3$. Then, the bundle $P$ is pulled back from $\mathbb{R}^3$ and so is trivial. However, the normal sphere bundle $\mathbb{S}^2(N)$ over $N$ is diffeomorphic to $\mathbb{T}^4 \times \mathbb{S}^2$ and the bundle $\mathfrak{g}_P \vert_{\mathbb{S}^2(N)}\cong \underline{\mathbb{R}} \oplus H$, where $H$ is the Hopf bundle over $\mathbb{S}^2$. Our example of a $G_2$-monopole on $P$ with a charge $1$ Dirac type singularity along $N$ comes from lifting a singular monopole on $\mathbb{S}^3$ and making use of lemma \ref{lem:reduction}.\\
Denote by $p_1$ and $e$ the Pontryagin and Euler class of $\mathfrak{g}_P$ respectively. As $\mathfrak{g}_P$ is trivial these are both zero in $H^4(M,\mathbb{Z})$ and $H^3(M,\mathbb{Z})$ respectively. On the other hand, we have the Poincar\'e dual of the vanishing locus of a section of $\mathfrak{g}_P$, for example $PD[\Phi^{-1}(0)] \in H^3_{cs}(M, \mathbb{Z})$, where $\Phi$ is the Higgs field of the monopole. Notice that in this case $\Phi^{-1}(0)=\mathbb{T}^4 \times \lbrace 0 \rbrace$ is coassociative. Moreover, the classes $e=0 \in H^3(M, \mathbb{Z})$ and $PD[\Phi^{-1}(0)] \in H^3_{cs}(M,\mathbb{Z})$ are related as follows. In the exact sequence
$$H^2(\mathbb{S}^2(N) , \mathbb{Z}) \xrightarrow{i} H^3_{cs}(M, \mathbb{Z}) \xrightarrow{j} H^3(M, \mathbb{Z}), $$
the class $PD[\Phi^{-1}(0)]$ maps through $j$ to $e=0$, and so is in the kernel of $j$. By exactness, $PD[\Phi^{-1}(0)]$ is then determined by a class in $H^2(\mathbb{S}^2(N), \mathbb{Z})$, which as we have seen in this case is $c_1(H)$, i.e. the first Chern class of the pullback of the Hopf bundle over the two sphere around $\lbrace \infty \rbrace$.


\subsection*{The main problem for future work}

The example above illustrates the topological invariants involved in setting up the problem for monopoles on $(X, \varphi)$ with Dirac singularities along a coassociative $N$. We fix a class $\alpha \in H^2(\mathbb{S}^2(N) , \mathbb{Z})$, and a principal $SO(3)$-bundle $P \rightarrow M$, such that $\mathfrak{g}_P \vert_{\mathbb{S}^2(N)} \cong \underline{\mathbb{R}} \oplus L^2$, where $L$ is a complex line bundle with $c_1(L)=\alpha$. The author intends to come back to the problem of constructing monopoles on $P$ which extend to $X$ with Dirac type singularities along $N$.

=======================================================

\end{document}